\documentclass{amsart}

\usepackage{amsmath,amssymb}
\usepackage{amsthm}
\usepackage[singlelinecheck=off]{caption}
\usepackage{microtype}
\usepackage{graphicx}
\usepackage{algorithm,algpseudocode}
\usepackage{mathrsfs}
\usepackage{epstopdf}
\usepackage{float}
\usepackage{url}
\usepackage{fancyhdr}
\usepackage{subfig}
\usepackage{epstopdf}

\makeatletter
\def\BState{\State\hskip-\ALG@thistlm}
\makeatother

\usepackage{epstopdf}
\usepackage{nicefrac}
\newtheorem{theorem}{Theorem}

\algblock{Input}{EndInput}
\algnotext{EndInput}
\algblock{Output}{EndOutput}
\algnotext{EndOutput}

\newcommand{\inp}{\mathrm{I}}
\newcommand{\out}{\mathrm{O}}
\newcommand\bbR{\mathbb{R}}
\newcommand\bbC{\mathbb{C}}

\newcommand\bbS{\mathbb{S}}

\newcommand{\eff}{\mathrm{eff}}
\newcommand{\test}{\mathrm{test}}

\newcommand\NN{\mathrm{NN}}
\newcommand\eps{\epsilon}

\title{SwitchNet: a neural network model for forward and inverse scattering problems}

\author{Yuehaw Khoo}

\address{Department of Mathematics, Stanford University, Stanford, CA 94305. }

\email{ykhoo@stanford.edu}

\author{Lexing Ying}

\address{Department of Mathematics and ICME, Stanford University, Stanford, CA 94305.
  Facebook AI Research, Menlo Park, CA 94025.
}

\email{lexing@stanford.edu}

\begin{document}

\begin{abstract}
  We propose a novel neural network architecture, SwitchNet, for solving the wave equation based
  inverse scattering problems via providing maps between the scatterers and the scattered field
  (and vice versa). The main difficulty of using a neural network for this problem is that a
  scatterer has a global impact on the scattered wave field, rendering typical convolutional
  neural network with local connections inapplicable. While it is possible to deal with such a
  problem using a fully connected network, the number of parameters grows quadratically with the
  size of the input and output data. By leveraging the inherent low-rank structure of the scattering
  problems and introducing a novel switching layer with sparse connections, the SwitchNet
  architecture uses much fewer parameters and facilitates the training process.  Numerical
  experiments show promising accuracy in learning the forward and inverse maps between the
  scatterers and the scattered wave field.
\end{abstract}

\maketitle
\section{Introduction}

In this paper, we study the forward and inverse scattering problems via the use of artificial
neural networks (NNs). In order to simplify the discussion, we focus on the time-harmonic acoustic
scattering in two dimensional space. The inhomogeneous media scattering problem with a fixed
frequency $\omega$ is modeled by the Helmholtz operator
\begin{equation}
  \label{helmholtz 0}
  Lu := \bigg(-\Delta - \frac{\omega^2}{c^2(x)} \bigg) u,
\end{equation}
where $c(x)$ is the velocity field. In many settings, there exists a known background velocity field
$c_0(x)$ such that $c(x)$ is identical to $c_0(x)$ except in a compact domain $\Omega$. By
introducing the {\em scatterer} $\eta(x)$ compactly supported in $\Omega$
\begin{equation}
\label{eta def}
\frac{\omega^2}{c(x)^2} =  \frac{\omega^2}{c_0(x)^2} + \eta(x),
\end{equation}
we can equivalently work with $\eta(x)$ instead of $c(x)$. Note that in this definition $\eta(x)$
scales quadratically with the frequency $\omega$. However, as $\omega$ is assumed to be fixed
throughout this paper, this scaling does not affect any discussion below.


In many real-world applications, $\eta(\cdot)$ is unknown. The task of the inverse problem is to
recover $\eta(\cdot)$ based on some observation data $d(\cdot)$. The observation data $d(\cdot)$ is
often a quantity derived from the Green's function $G=L^{-1}$ of the Helmholtz operator $L$ and,
therefore, it depends closely on $\eta(\cdot)$. This paper is an exploratory attempt of constructing
efficient approximations to the forward map $\eta \rightarrow d$ and the inverse map
$d\rightarrow\eta$ using the modern tools from machine learning and artificial intelligence. Such
approximations are highly useful for the numerical solutions of the scattering problems: an
efficient map $\eta\rightarrow d$ provides an alternative to expensive partial differential equation
(PDE) solvers for the Helmholtz equation; an efficient map $d \rightarrow \eta$
is more valuable as it allows us to solve the inverse problem of determining the scatterers from the
scattering field, without going through the usual iterative process.

In the last several years, deep neural network has become the go-to method in computer vision, image
processing, speech recognition and many other machine learning applications \cite{lecun2015deep,
  schmidhuber2015deep, hinton2006reducing, goodfellow2016deep}. More recently, methods based on NN
have also been applied to solving PDEs. Based on the way that the NN is used, these methods for
solving the PDE can be roughly separated into two different categories. For the methods in the first
category
\cite{lagaris1998artificial,rudd2015constrained,carleo2017solving,han2018solving,Khoo2018,weinan2018deep},
instead of specifying the solution space via the choice of basis (as in finite element method or
Fourier spectral method), NN is used for representing the solution. Then an optimization problem,
for example an variational formulation, is solved in order to obtain the parameters of the NN and
hence the solution to the PDE.  Similar to the use of an NN for regression and classification
purposes, the methods in the second category such as
\cite{long2017pde,han2017deep,khoo2017solving,fan2018multiscale} use an NN to learn a map that goes
from the coefficients in the PDE to the solution of the PDE. As in machine learning, the
architecture design of an NN for solving PDE usually requires the incorporation of the knowledge
from the PDE domain such that the NN architecture is able to capture the behavior of the solution
process. Despite the abundance of the works in using the NN for solving PDE, none of the above
mentioned methods have tried to obtain the solution to the wave equation.

This paper takes a deep learning approach to learn both the forward and inverse maps. For the
Helmholtz operator \eqref{helmholtz 0}, we propose an NN architecture for determining the forward
and inverse maps between the scatterer $\eta(\cdot)$ and the observation data $d(\cdot)$ generated
from the scatterer. Although this task looks similar to the computer vision problems such as image
segmentation, denoising, and super-resolution where the map between the two images has to be
determined, the nature of the map in our problem is much more complicated. In many image processing
tasks, the value of a pixel at the output generally only depends on a neighborhood of that pixel at
the input layer. However, for the scattering problems, the input and output are often defined on
different domains and, due to wave propagation, each location of the scatterer can influence every
point of the scattered field. Therefore, the connectivity in the NN has to be wired in a non-local
fashion, rendering typical NN with local connectivity insufficient. This leads to the development of
the proposed {\em SwitchNet}. The key idea is the inclusion of a novel low-complexity {\em switch
  layer} that sends information between all pairs of sites effectively, following the ideas from
butterfly factorizations \cite{li2015butterfly}. The same factorization was used earlier in the
architecture proposed \cite{li2018butterfly}, but the network weights there are hardcoded and not
trainable.

The paper is organized as followed. In Section \ref{section: prelim}, we discuss about some
preliminary results concerning Helmholtz equation. In Section \ref{section: NN archi}, we study the
so called far field pattern of the scattering problem, where the sources and receivers can be
regarded as placed at infinity. We propose SwitchNet to determine the maps between the far field
scattering pattern and the scatterer. In Section \ref{section: NN seis}, we turn to the setting of a
seismic imaging problem. In this problem, the sources and receivers are at a finite distance, but
yet well-separated from  the scatterer.



\section{Preliminary} \label{section: prelim}

The discussion of this paper shall focus on the two-dimensional case. Here, we summarize the
mathematical tools and notations used in this paper. As mentioned above, the scatterer $\eta(x)$ is
compactly supported in a domain $\Omega$, whose diameter is of $O(1)$. For example, one can think of
$\Omega$ to be the unit square centered at the origin. In \eqref{helmholtz 0}, the Helmholtz
operator is defined on the whole space $\bbR^2$ with the radiative (Sommerfeld) boundary condition
\cite{Colton2013} specified at infinity. Since the scatterer field $\eta(x)$ is localized in
$\Omega$, it is convenient to truncate the computation domain to $\Omega$ by imposing the
\emph{perfectly matched layer} \cite{berenger1994perfectly} that approximates the radiative boundary
condition.

In a typical numerical solution of the Helmholtz operator, $\Omega$ is discretized by a Cartesian
grid $X\subset\Omega$ at the rate of a few points per wavelength. As a result, the number of grid
points $N$ per dimension is proportional to the frequency $\omega$. We simply use $\{x\}_{x\in X}$
to denote the discretization points of this $N\times N$ grid $X$. The Laplacian operator $-\Delta$
in the Helmholtz operator is typically discretized with local numerical schemes, such as the finite
difference method \cite{Larsson2009}. Via this discretization, we can consider the scatterer field
$\eta$, discretized at the points in $X$, as a vector in $\mathbb{R}^{N^2}$ and the Helmholtz
operator $L$ as a matrix in $\mathbb{C}^{N^2\times N^2}$.


Using the background velocity field $c_0(x)$, we first introduce the background Helmholtz operator
$L_0 = -\Delta - \omega^2/c_0^2$. With the help of $L_0$, one can write $L$ in a perturbative way as
\begin{equation}  \label{helmholtz}
  L = L_0 - E, \quad
  E = \text{diag}(\eta),
\end{equation}
where $E$ is viewed as a perturbation. By introducing the background Green's function
\begin{equation}
  G_0 := L_0^{-1},
\end{equation}
one can write down a formal expansion for the Green's function $G=L^{-1}$ of the $\eta$-dependent
Helmholtz operator $L$:
\begin{eqnarray}
  G & = & (L_0(I-G_0E))^{-1}\cr
  & \sim & (I + G_0E + G_0E G_0E + \cdots) G_0 \cr
  & \sim & G_0 + G_0EG_0 + G_0E  G_0EG_0 + \cdots \cr
  &:=& G_0 + G_1 + G_2 + \cdots,
\end{eqnarray}
which is valid when the scatterer field $\eta(x)$ is sufficiently small. The last line of the above
equation serves as the definition of the successive terms of the expansion ($G_1$, $G_2$, and so
on). As $G_0$ can be computed from the knowledge of the background velocity field $c_0(x)$, most
data gathering processes (with appropriate post-processing) focus on the difference $G-G_0 = G_1 +
G_2 +\cdots$ instead of $G$ itself.

A usual experimental setup consists of a set of {\em sources} $S$ and a set of {\em receivers} $R$:
\[
S = \{s\}_{s\in S},\quad R = \{r\}_{r\in R}.
\]
The data gathering process usually involves three steps: (1) impose an external force or incoming
wave field via some sources, (2) solve for the scattering field either computationally or
physically, (3) gather the data with receivers at specific locations or directions. The second step
is modeled by the difference of the Green's function $G-G_0$, as we mentioned above. As for the
other steps, it is convenient at this point to model the first step with a source-dependent operator
$\Pi_S$ and the third one with a receiver-dependent operator $\Pi_R$. We shall see later how these
operators are defined in more concrete settings. By putting these components together, one can set
the observation data $d$ abstractly as
\begin{equation}\label{ddefined}
  d = \Pi_R (G-G_0) \Pi_S =
  \Pi_R (G_0EG_0 + G_0E  G_0EG_0 + \cdots) \Pi_S =
  (\Pi_R G_0) (E + EG_0E + \cdots) (G_0 \Pi_S).
\end{equation}



In this paper, we focus on two scenarios: far field pattern and seismic imaging. We start with far
field pattern first to motivate and introduce SwitchNet. We then move on to the seismic case by
focusing on the main differences.



\section{SwitchNet for far field pattern}\label{section: NN archi}

\subsection{Problem setup}\label{section:FFsetup}

In this section, we consider the problem of determining the map from the scatterer to the far field
scattering pattern, along with its inverse map. Without loss of generality, we assume that the
diameter of the domain $\Omega$ is of $O(1)$ after appropriate rescaling. The background velocity
$c_0(x)$ is assumed to be 1 since the far field pattern experiments are mostly performed in free
space.

In this problem, both the sources and the receivers are indexed by a set of unit directions in
$\bbS^1$. The source associated with a unit direction $s\in S\subset\bbS^1$ is an incoming plane
waves $e^{i\omega s\cdot x}$ pointing at direction $s$. It is well known that the scattered wave
field, denoted by $u_s(x)$, at a large distance takes the following form \cite{Colton2013}
\[
u_s(x) =\frac{e^{i\omega |x|}}{\sqrt{|x|}} \left( u^\infty_s\left(\frac{x}{\vert x \vert}\right) + o(1) \right),
\]
where the function $u^\infty_s(\cdot)$ is defined on the unit circle $\bbS^1$. The receiver at
direction $r\in R\subset \bbS^1$ simply records the quantity $u^\infty_s(r)$ for each $s$. The set of
observation data $d$ is then defined to be
\[
d(r s) = u^\infty_s(r).
\]
Figure \ref{fig:far} provides an illustration of this experimental setup. Henceforth, we assume that
both $R$ and $S$ are chosen to be a set of uniformly distributed directions on $\bbS^1$. Their size,
denoted by $M$, typically scales linearly with frequency $\omega$.

\begin{figure}[!ht]
  \centering
  \includegraphics[height=0.25\textwidth]{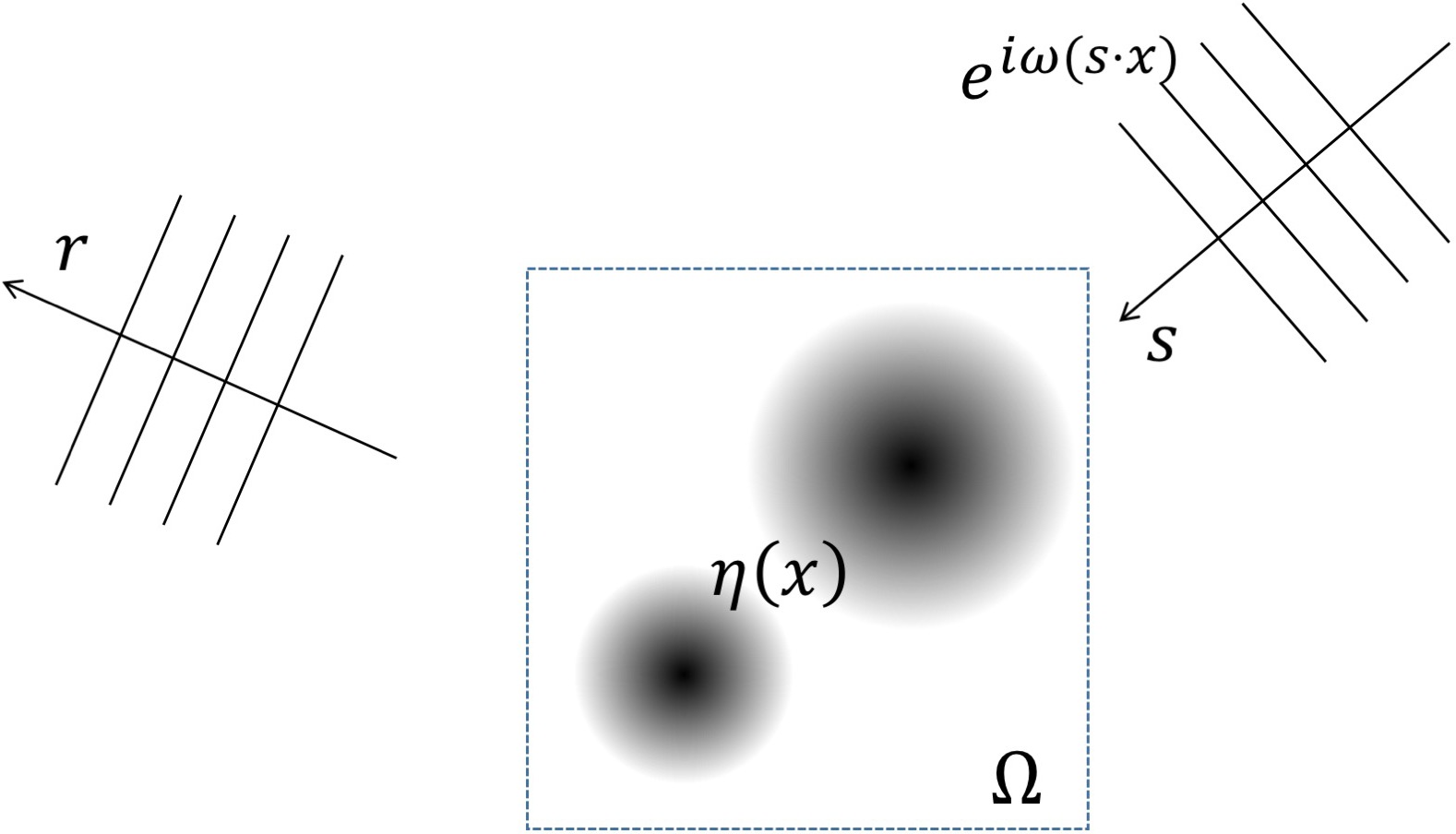}
  \caption{Illustration of the incoming and outgoing waves for a far field pattern problem.  The
    scatterer $\eta(x)$ is compactly supported in the domain $\Omega$. The incoming plane wave
    points at direction $s$. The far field pattern is sampled at each receiver direction $r$.
  }\label{fig:far}
\end{figure}

This data gathering process can be put into the framework of \eqref{ddefined}. First, one can think
of the source prescription as a limiting process that produces in the limit the incoming wave
$e^{i\omega s\cdot x}$. The source can be considered to be located at the point $-s\rho$ for the
direction $s\in\bbS^1$ with the distance $\rho\in\bbR^+$ going to infinity. In order to compensate
the geometric spreading of the wave field and also the phase shift, the source magnitude is assumed
to scale like $\sqrt{\rho} e^{-i\omega \rho}$ as $\rho$ goes to infinity. Under this setup, we have
\begin{eqnarray}\label{slimit}
  &\ &\lim_{\rho\rightarrow \infty} (G_0 \Pi_S)(x, s)\cr
  &=& \lim_{\rho\rightarrow \infty} (1/\sqrt{\rho}) e^{i\omega |x-(-s\rho)|} \sqrt{\rho} e^{-i\omega \rho}\cr
  &=& \lim_{\rho\rightarrow \infty} (1/\sqrt{\rho}) e^{i\omega (\rho + s\cdot x)} \sqrt{\rho} e^{-i\omega \rho}\cr
  &=& e^{ i\omega s\cdot x}.
\end{eqnarray}
Similarly, one can also regard the receiver prescription as a limiting process as well. The receiver
is located at point $r\rho'$ for a fixed unit direction $r\in\bbS^1$ with $\rho'\in\bbR^+$ going to
infinity. Again in order to to compensate the geometric spreading and the phase shift, one scales
the received signal with $\sqrt{\rho'} e^{-i\omega \rho'}$. As a result, we have
\begin{eqnarray}\label{rlimit}
  &\ &\lim_{\rho'\rightarrow \infty} (\Pi_R G_0)(r, x)\cr
  &=& \lim_{\rho'\rightarrow \infty} (1/\sqrt{\rho'}) e^{i\omega |r\rho' - x|} \sqrt{\rho'} e^{-i\omega \rho'}\cr
  &=& \lim_{\rho'\rightarrow \infty}  (1/\sqrt{\rho'}) e^{i\omega (\rho' - r\cdot x)} \sqrt{\rho'} e^{-i\omega \rho'}\cr
  &=& e^{-i\omega r\cdot x}.
\end{eqnarray}
In this limiting setting, one redefine the observation data as 
\begin{equation}
  d =\underset{\rho,\rho'\rightarrow \infty}\lim (\Pi_R G_0) (E+EG_0E+\cdots) (G_0\Pi_S).
\end{equation}
Now taking the two limits \eqref{slimit} and \eqref{rlimit} under consideration, one arrives at the
following representation of the observation data $d(r,s)$ for $r\in R$ and $s\in S$
\begin{equation}
\label{perturbation series}
d(r,s) = \sum_{x\in X} \sum_{y\in X} e^{-i\omega r\cdot x} (E+EG_0E+\cdots)(x,y) e^{i\omega s\cdot y}.
\end{equation}

\subsection{Low-rank property} \label{section: FIO}

The intuition behind the proposed NN architecture comes from examining \eqref{perturbation series}
when $E$ (or $\eta$) is small. In such a situation, we simply retain the term that is linear in $E$.
Using the fact that $E=\text{diag}(\eta)$, \eqref{perturbation series} becomes
\[
d(r,s) \approx  \sum_{x\in X} e^{i\omega (s-r)\cdot x} \eta(x)
\]
for $r\in R\subset \mathbb{S}^1$ and $s\in S\subset \mathbb{S}^1$.  This linear map takes $\eta(x)$
defined on a Cartesian grid $X\subset \Omega$ to $d(r,s)$ defined on yet another Cartesian grid
$R\times S\subset\mathbb{S}^1\times\mathbb{S}^1$.  Recalling that both $R$ and $S$ are of size $M$
and working with a vectorized $d\in \mathbb{C}^{M^2}$, we can write the above equation compactly as
\begin{equation}
  d \approx A\eta,
\end{equation}
where the element of the matrix $A\in\mathbb{C}^{M^2 \times N^2}$ at $(r,s)\in R\times S$ and $x\in
X$ is given by
\begin{equation}
  A(rs,x) = \exp(i\omega(s-r)\cdot x).
\end{equation}

The following theorem concerning the matrix $A$ plays a key role in the design of our NN. Let us
first partition $\Omega$ uniformly into $\sqrt{P_X}\times\sqrt{P_X}$ Cartesian squares of
side-length equal to $1/\sqrt{\omega}$. Here we assume that $\sqrt{P_X}$ is an integer. Note that,
since the diameter of $\Omega$ is of $O(1)$, $\sqrt{P_X} \approx \sqrt{\omega}$. This naturally
partitions the set of grid points $X$ into $P_X$ subgroups depending on which square each point
belongs to. We shall denote these subgroups by $X_0,\ldots, X_{P_X-1}$. Similarly, we also partition
$\mathbb{S}^1\times\mathbb{S}^1$ uniformly (in the angular parameterization) into
$\sqrt{P_D}\times\sqrt{P_D}$ squares $D_0,\ldots,D_{P_D-1}$ of side-length equal to
$1/\sqrt{\omega}$. $\sqrt{P_D}$ is also assumed to be an integer and obviously $\sqrt{P_D}\approx
\sqrt{\omega}$. This further partitions the set $R\times S$ into $P_D$ subgroups depending on which
square they belong to. We shall denote these subgroups by $D_0,\ldots,
D_{P_D-1}$. Fig. \ref{fig:thm1} illustrates the partition for $\sqrt{P_X}=\sqrt{P_D}=4$.

\begin{figure}[h]
  \centering \includegraphics[width=0.6\textwidth]{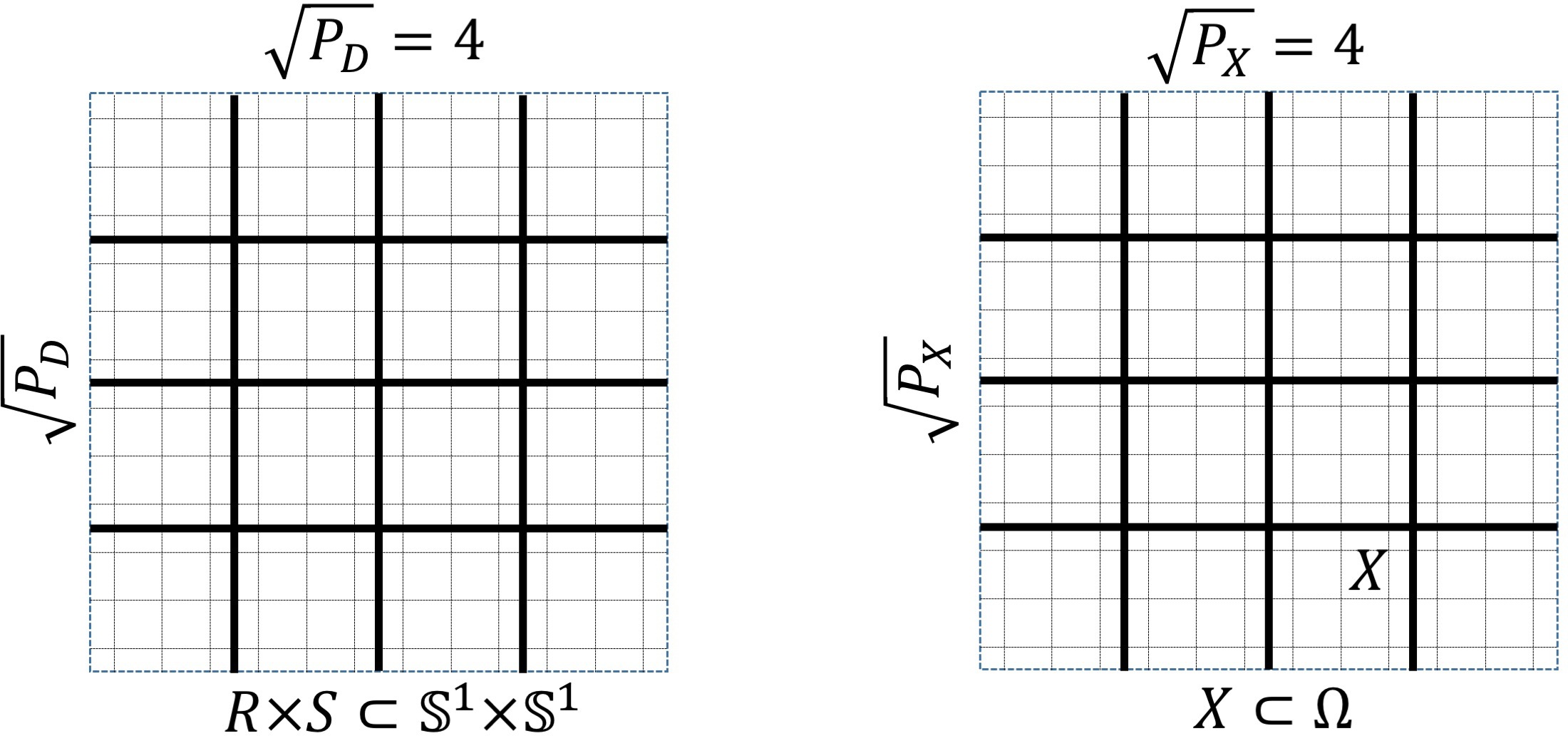}
  \caption{ Illustration of the partitions used in Theorem \ref{theorem: low-rank}. The fine grids
    stand for the Cartesian grids $X$ and $R\times S$. The bold lines are the boundary of the
    squares of the partitions.  }\label{fig:thm1}
\end{figure}

\begin{theorem}
\label{theorem: low-rank}
For any $D_i$ and $X_j$, the submatrix
\begin{equation}\label{Aij}
  A_{ij} := [A(rs,x)]_{(r,s)\in D_i, x\in X_j}
\end{equation}
is numerically low-rank.
\end{theorem}

\begin{proof}
The proof of this theorem follows the same line of argument in
\cite{candes2009fast,ying2009sparse,li2017interpolative} and below we outline the key idea.  Denote
the center of $D_i$ by $(r_i,s_i)$ and the center of $X_j$ by $x_j$. For each $(r,s)\in D_i$ and
$x\in X_j$, we write
\begin{multline}
  \label{centered factorization}
  \exp(i\omega(s-r)\cdot x) = \exp(i\omega((s-r) - (s_i-r_i))\cdot (x-x_j)) \cdot \cr
  \exp(i\omega(s_i-r_i)\cdot x)\cdot \exp(i \omega (s-r)\cdot x_j)\cdot \exp(-i\omega (s_i-r_i)\cdot x_j).
\end{multline}
Note that for fixed $D_i$ and $X_j$ each of the last three terms is either a constant or depends
only on $x$ or $(r,s)$. As a result, $\exp(i\omega(s-r)\cdot x) $ is numerically low-rank if and
only if the first term $\exp(i\omega((s-r) - (s_i-r_i))\cdot (x-x_j))$ is so. Such a low-rank
property can be derived from the conditions concerning the side-lengths of $D_i$ and $X_j$. More
precisely, since $(r,s)$ resides in $D_i$ with center $(r_i,s_i)$, then
\begin{equation}
  \vert (s-r) - (s_i-r_i) \vert \leq \frac{1}{\sqrt{\omega}}.
\end{equation}
Similarly as $x$ resides in $X_j$ with center $x_j$, then
\begin{equation}
  \vert x - x_j\vert \leq \frac{1}{\sqrt{\omega}}.
\end{equation}
Multiplying these two estimates results in the estimate
\begin{equation}
  \omega \vert ((s-r) - (s_i-r_i))\cdot (x-x_j))\vert \leq 1
\end{equation}
for the phase of $\exp(i\omega((s-r) - (s_i-r_i))\cdot (x-x_j))$. Therefore,
\begin{equation}
  \exp(i\omega((s-r) - (s_i-r_i))\cdot (x-x_j))
\end{equation}
for $(r,s)\in D_i$ or $x\in X_j$ is non-oscillatory and hence can be approximated effectively by
applying, for example, Chebyshev interpolation in both the $(r,s)$ and $x$ variables. Since the
degree of the Chebyshev polynomials only increases poly-logarithmically with respect to the desired
accuracy, $\exp(i\omega((s-r) - (s_i-r_i))\cdot (x-x_j))$ is numerically low-rank by
construction. This proves that the submatrix $A_{ij}$ defined in \eqref{Aij} is also numerically
low-rank.
\end{proof}

\subsection{Matrix factorization} \label{section: fac}

In this subsection, we show that Theorem \ref{theorem: low-rank} guarantees a {\em low-complexity}
factorization of the matrix $A$. Let the row and column indices of $A\in\mathbb{C}^{M^2 \times N^2}$
be partitioned into index sets $\{D_i\}_{i=0}^{P_D-1}$ and $\{X_j\}_{j=0}^{P_X-1}$, respectively, as
in Theorem \ref{theorem: low-rank}. To simplify the presentation, we assume $P_X=P_D=P$, $\vert X_0
\vert = \cdots \vert X_{P-1} \vert = N^2/P$, and $\vert D_0\vert=\cdots=\vert D_{P-1}\vert=M^2/P$.

Since the submatrix
\[
A_{ij} := [A(rs,x)]_{rs\in D_i, x\in X_j}
\]
is numerically low-rank, assume that 
\begin{equation}\label{AUV}
  A_{ij}\approx U_{ij} V_{ij}^*,
\end{equation} 
where $U_{ij}\in \bbC^{M^2/P\times t}$ and $V_{ij}\in \bbC^{N^2/P\times t}$. Here $t$ can be taken
to be the maximum of the numerical ranks of all submatrices $A_{ij}$. Theorem \ref{theorem:
  low-rank} implies that $t$ is a small constant.

By applying \eqref{AUV} to each block $A_{ij}$, $A$ can be approximated by
\begin{equation}
  \begin{bmatrix}
    U_{00} V_{00}^* & U_{01} V_{01}^* & \cdots & U_{0(P-1)} V_{0(P-1)}^* \\ 
    U_{10} V_{10}^* & U_{11} V_{11}^* & \cdots & U_{1(P-1)} V_{1(P-1)}^* \\ 
    \vdots & \ & \ddots & \vdots \\
    U_{(P-1)0} V_{(P-1)0}^* & U_{(P-1)1} V_{(P-1)1}^*  & \cdots & U_{(P-1)(P-1)} V_{(P-1)(P-1)}^* \\
  \end{bmatrix}.
  \label{butterfly mid}
\end{equation}
The next step is to write \eqref{butterfly mid} into a factorized form. First, introduce $U_i$ and
$V_j$
\begin{equation}
  \label{UV def}
  U_i = \begin{bmatrix} U_{i0}, U_{i1}, \cdots, U_{i(P-1)} \end{bmatrix}\in \mathbb{C}^{M^2/P \times
    tP},\quad V_j = \begin{bmatrix} V_{0j}, V_{1j}, \cdots, V_{(P-1)j} \end{bmatrix}\in
  \mathbb{C}^{N^2/P \times tP},
\end{equation}
and define in addition
\begin{equation}\label{UV defmore}
  U = \begin{bmatrix} U_0 & & & \\ & U_1 & & \\ & & \ddots & \\ & & & U_{P-1}
  \end{bmatrix}\in\mathbb{C}^{M^2 \times P^2t},\quad
  V^* = \begin{bmatrix} {V_0}^* & & & \\ & {V_1}^* & & \\ & & \ddots & \\ & & & {V_{P-1}}^*
  \end{bmatrix}\in\mathbb{C}^{P^2t\times N^2},
\end{equation}
In addition, introduce 
\begin{eqnarray}
  \label{S def}
  \Sigma = 
  \begin{bmatrix}
    \Sigma_{00} & \Sigma_{01} & \cdots & \Sigma_{0(P-1)} \\ 
    \Sigma_{10} & \Sigma_{11} & \cdots & \Sigma_{1(P-1)} \\ 
    \vdots & \ & \ddots & \vdots \\
    \Sigma_{(P-1)0} & \Sigma_{(P-1)1}  & \cdots & \Sigma_{(P-1)(P-1)}\\
  \end{bmatrix}\in\mathbb{C}^{P^2t \times P^2t},
\end{eqnarray}
where the submatrix $\Sigma_{ij} \in \mathbb{C}^{Pt\times Pt}$ itself is a $P\times P$ block matrix with
blocks of size $t\times t$. $\Sigma_{ij}$ is defined to be zero everywhere except being the identity
matrix at the $(j,i)$-th $t \times t$ block. In order to help understand the NN architecture
discussed below sections, it is imperative to understand the meaning of $\Sigma$. Let us assume for
simplicity that $t=1$. Then for an arbitrary vector $z\in \mathbb{C}^{P^2}$, $\Sigma z$ essentially
performs a ``switch'' that shuffles $z$ as follows
\begin{equation}
  (\Sigma z)(jP+i) = z(iP+j),\quad i,j=0,\ldots,P-1.
\end{equation} 

With the above definitions for $U$, $V$, and $\Sigma$, the approximation in \eqref{butterfly mid} can be
written compactly as
\begin{equation} \label{AUSV}
  A \approx U \Sigma V^*.
\end{equation}
Notice that although $A$ has $M^2 \times N^2$ entries, using the factorization \eqref{AUSV}, $A$ can
be stored using $tP(M^2 + P + N^2)$ entries. In this paper, $P \approx \max(M,N)$ and $M$ and $N$
are typically on the same order. Therefore, instead of $O(N^4)$, one only needs $O(N^3)$ entries to
parameterize the map $A$ approximately using \eqref{AUSV}. Such a factorization is also used in
\cite{li2015butterfly} for the compression of Fourier integral operators.

We would like to comment on another property that may lead to further reduction in the parameters
used for approximating $A$. Let us focus on any two submatrices $A_{ij}$ and $A_{ik}$ of $A$. For
two regions $X_j$ and $X_k$, where the center of $X_j$ and $X_k$ are $x_{j}$ and $x_{k}$
respectively, $X_k=X_j + (x_{k}-x_{j})$. Let $(r,s)\in D_i$. For $x\in X_j$ and $x' = x +
(x_{k}-x_{j})\in X_k$, we have
\begin{eqnarray}
  \exp(i\omega(s-r)\cdot x)  &=&  g_1(r,s) h((r,s),x),\cr
  \exp(i\omega(s-r)\cdot x') &=&  g_2(r,s) h((r,s),x),
\end{eqnarray}
where 
\begin{equation}
  g_1(r,s) = \exp(i\omega(s-r)\cdot x_{j}),\quad
  g_2(r,s) = \exp(i\omega(s-r)\cdot x_{k}),\quad
  h((r,s),x) = \exp(i\omega(s-r)\cdot (x-x_{j})).
\end{equation}
Therefore, the low-rank factorizations of $A_{ij}$ and $A_{ik}$ are solely determined by the
factorization of $h(rs,x)$. This implies that it is possible to construct low-rank factorizations
for $A_{ij}$ and $A_{ik}$:
\begin{equation}
  A_{ij} \approx U_{ij} V^*_{ij},\quad
  A_{ik} \approx U_{ik} V^*_{ik},
\end{equation}
such that $V^*_{ij} = V^*_{ik}$. Since this is true for all possible $j,k$, one can pick low-rank
factorizations so that $V_0=V_1=\cdots=V_{P-1}$.

As a final remark in this section, this low complexity factorization \eqref{AUSV} for $A$ can be
easily converted to one for $A^*$ since
\begin{equation}\label{AtVSU}
  A^* \approx V \Sigma^* U^*,
\end{equation}
where $U,\Sigma,V$ are provided in \eqref{UV def}, \eqref{UV defmore}, and \eqref{S def}.

\subsection{Neural networks} \label{section: switchnet}

Based on the low-rank property of $A$ in Section \ref{section: FIO} and its low-complexity
factorization in Section \ref{section: fac}, we propose new NN architectures for representing the
inverse map $d \rightarrow \eta$ and the forward map $\eta \rightarrow d$.

\subsubsection{NN for the inverse map $d\rightarrow \eta$.}
As pointed out earlier, $d \approx A \eta$ when $\eta$ is sufficiently small. The usual filtered
back-projection algorithm \cite{Natterer2001} solves the inverse problem $d\rightarrow \eta$ via
\begin{equation}
  \label{adjoint and deconv}
  \eta \approx (A^* A + \eps I)^{-1} A^* d,
\end{equation}
where $\eps$ is the regularization parameter. In the far field pattern problem, $(A^* A + \eps
I)^{-1} $ can be understood as a deconvolution operator. To see this, a direct calculation reveals
that
\begin{eqnarray}
  \label{full kernel}
  (A^*A)(x,y) = \sum_{rs\in R\times S} e^{i\omega(s-r)\cdot y} e^{-i\omega(s-r)x}
  = \sum_{rs\in R\times S} e^{-i\omega(s-r)(x-y)}
\end{eqnarray}
for $x,y\in X$. \eqref{full kernel} shows that $A^* A$ is a translation-invariant convolution
operator. Therefore, the operator $(A^* A + \eps I)^{-1}$, as a regularized inverse of $A^* A$,
simply performs a deconvolution. In summary, the above discussion shows that in order to obtain
$\eta$ from the scattering pattern $d$ in the regime of small $\eta$, one simply needs to apply
sequentially to $d$
\begin{itemize} 
\item the operator $A^*$,
\item a translation-invariant filter that performs the deconvolution $(A^* A+\eps I)^{-1}$.
\end{itemize}


Although these two steps might be sufficient when $\eta$ is small, a nonlinear solution is needed
when $\eta$ is not so. For this purpose, we propose a nonlinear neural network {\em SwitchNet} for
the inverse map. There are two key ingredients in the design of SwitchNet.
\begin{itemize}
\item The first key step is the inclusion of a \texttt{Switch} layer that sends local information
  globally, as depicted in Figure \ref{figure: switch}.  The structure of the \texttt{Switch} layer
  is designed to mimic the matrix-vector multiplication of the operator $A^* \approx V \Sigma^* U^*$
  in \eqref{butterfly mid}. However unlike the fixed coefficients in \eqref{butterfly mid}, as an NN
  layer, the \texttt{Switch} layer allows for tunable coefficients and learns the right values for
  the coefficients from the training data. This gives the architecture a great deal of flexibility.
  
\item The second key step is to replace the linear deconvolution in the back-projection algorithm
  with a few convolution (\texttt{Conv}) layers. This enriches the architecture with nonlinear
  capabilities when approximating the nonlinear inverse map.
\end{itemize}

\begin{algorithm}[h]
  \caption{SwitchNet for the inverse map $d\rightarrow \eta$ of far field
    pattern.}\label{algorithm: inverse}
  \begin{algorithmic}[1]
	\Require $t,P_D,P_X,N,w,\alpha,L,d\in \mathbb{C}^{M\times M}$
	\Ensure $\eta\in \mathbb{C}^{N\times N}$
	\State $d_1\leftarrow  \texttt{Vect}[P_D](d)$
	\State $d_2\leftarrow  \texttt{Switch}[t,P_D,P_X,N^2](d_1)$
	\State $e_0\leftarrow \texttt{Square}[P_X](d_2)$
	\For  {$\ell$ from $0$ to $L-1$}
	\State  $e_{\ell+1}\leftarrow \texttt{Conv}[w,\alpha](e_{\ell})$
	\EndFor
	\State $\eta \leftarrow \texttt{Conv}[w,1](e_L)$
	\State \Return $\eta$
  \end{algorithmic}
\end{algorithm}

The pseudo-code for SwitchNet is summarized in Algorithm \ref{algorithm: inverse}. The input $d$ is
a $\bbC^{M\times M}$ matrix, while the output $\eta$ is a $\bbC^{N\times N}$ matrix. The first three
steps of Algorithm \ref{algorithm: inverse} mimics the application of the operator $A^* \approx V
\Sigma^* U^*$. The \texttt{Switch} layer does most of the work, while the \texttt{Vect} and
\texttt{Square} layers are simply operations that reshape the input and output data to the correct
matrix form at the beginning and the end of the \texttt{Switch} layer. In particular, \texttt{Vect}
groups the entries of the 2D field $d$ according to squares defined by the partition
$D_0,\ldots,D_{P_D-1}$ and \texttt{Square} does the opposite. The remaining lines of Algorithm
\ref{algorithm: inverse} simply apply the \texttt{Conv} layers with window size $w$ and channel
number $\alpha$.

These basic building blocks of SwitchNet are detailed in the following subsection. We also take the
opportunity to include the details of the pointwise multiplication \texttt{PM} layer that will be
used in later on.

\subsubsection{Layers for SwitchNet}\label{section: layer}
In this section we provide the details for the layers that are used in SwitchNet. Henceforth, we
assume that the entries of a tensor is enumerated in the Python convention, i.e., going through the
dimensions from the last one to the first. One operation that will be used often is a {\em reshape},
in which a tensor is changed to a different shape with the same number of entries and with the
enumeration order of the entries kept unchanged.

\begin{figure}[h]
  \centering \includegraphics[width=0.4\textwidth]{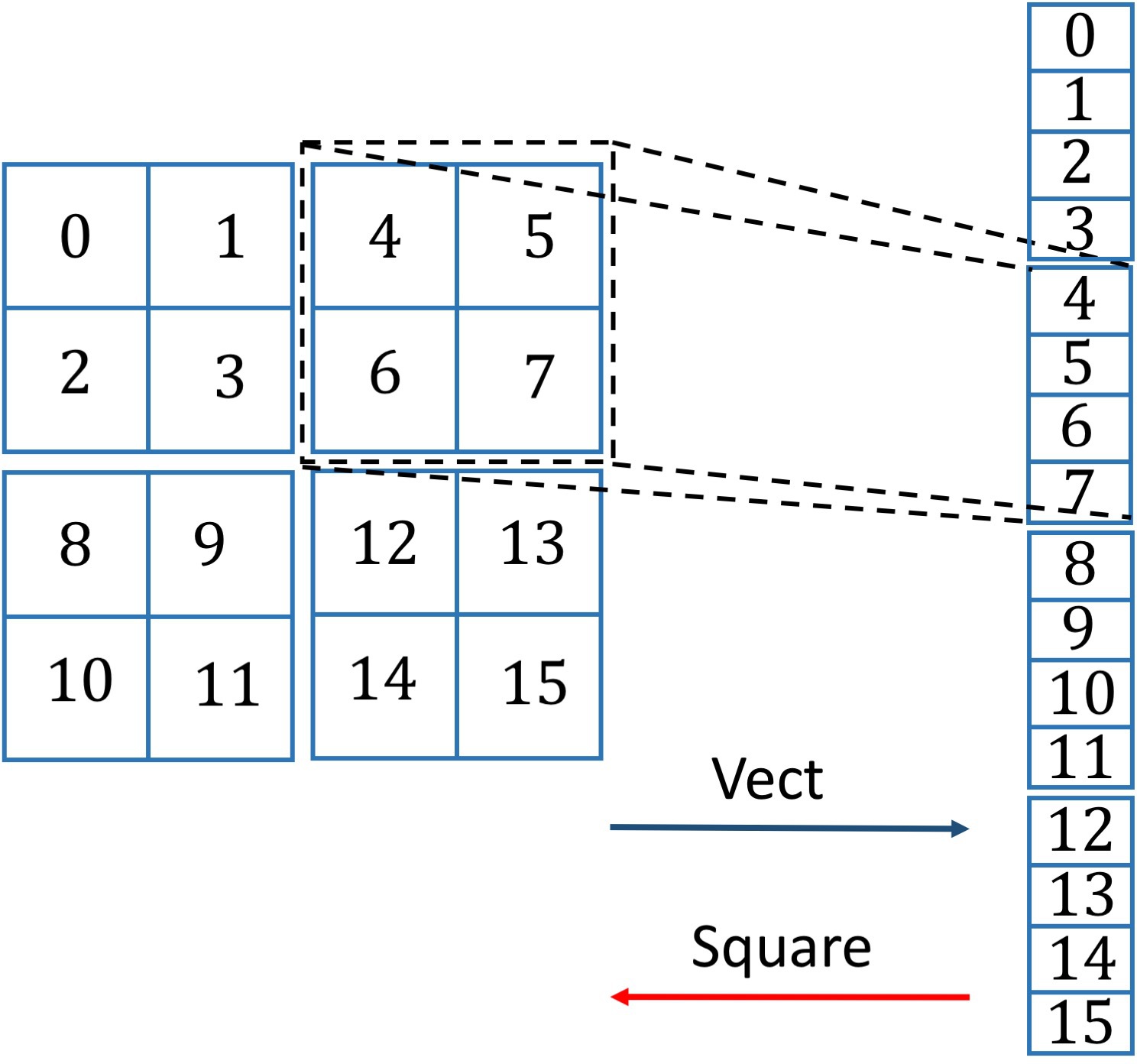}
  \caption{An illustration of the \texttt{Vect} and \texttt{Square} layers. The detail descriptions
    of the layers are provided in Section \ref{section: layer}. For the purpose of illustration we let $P=4$. The
    \texttt{Vect} layer vectorize a $4\times 4$ matrix on the left hand side according to the
    partitioning by $2\times 2$ blocks, to give the size 16 vector on the right hand side. The
    \texttt{Square} layer is simply the adjoint map of the \texttt{Vect} layer.  }\label{figure:
    vectorize}
\end{figure}

{\bf Vectorize layer.} $z_\out = \texttt{Vect}[P](z_\inp)$ with input $z_\inp \in \bbC^{n\times
  n}$. Henceforth we assume that $\sqrt{P}$ is an integer and $\sqrt{P}$ divides $n$. This operation
partitions $z_\inp$ into $\sqrt{P}\times\sqrt{P}$ square sub-blocks of equal size. Then each
sub-block is vectorized, and the vectorized sub-blocks are stacked together as a vector in
$\bbC^{n^2}$. Intuitively, these operations cluster the nearby entries in a sub-block together. The
details of the \texttt{Vect} layer are given in the following: \label{item: vect}
\begin{itemize}
\item Reshape the $z_\inp$ to a ${\sqrt P}\times\frac{n}{\sqrt P}\times{\sqrt P}\times\frac{n}{\sqrt P}$ tensor.
\item Swap the second and the third dimensions to get a
  ${\sqrt P}\times {\sqrt P}\times\frac{n}{\sqrt P}\times\frac{n}{\sqrt P}$ tensor.
\item Reshape the result to an $n^2$ vector and set it to $z_\out$.
  
\end{itemize}

{\bf Square layer.} $z_\out = \texttt{Square}[P](z_\inp)$ with input $z_\inp \in \bbC^{n^2}$, where
$\sqrt P$ is an integer. The output is $z_\out \in \mathbb{R}^{n\times n}$.  Essentially as the adjoint
operator of the $\texttt{Vect}$ layer, this layer fills up each square sub-block of the matrix
$z_\out$ with a segment of entries in $z_\inp$. The details are given as followed:
\begin{itemize}
\item Reshape the $z_\inp$ to a ${\sqrt P}\times {\sqrt P}\times\frac{n}{\sqrt P}\times\frac{n}{\sqrt P}$ tensor.
\item Swap the second and the third dimensions to get a
  ${\sqrt P}\times\frac{n}{\sqrt P}\times{\sqrt P}\times\frac{n}{\sqrt P}$ tensor.
\item Reshape the result to an $n\times n$ matrix and set it to $z_\out$.
\end{itemize}

\begin{figure}[h]
  \centering
  \includegraphics[width=0.7\textwidth]{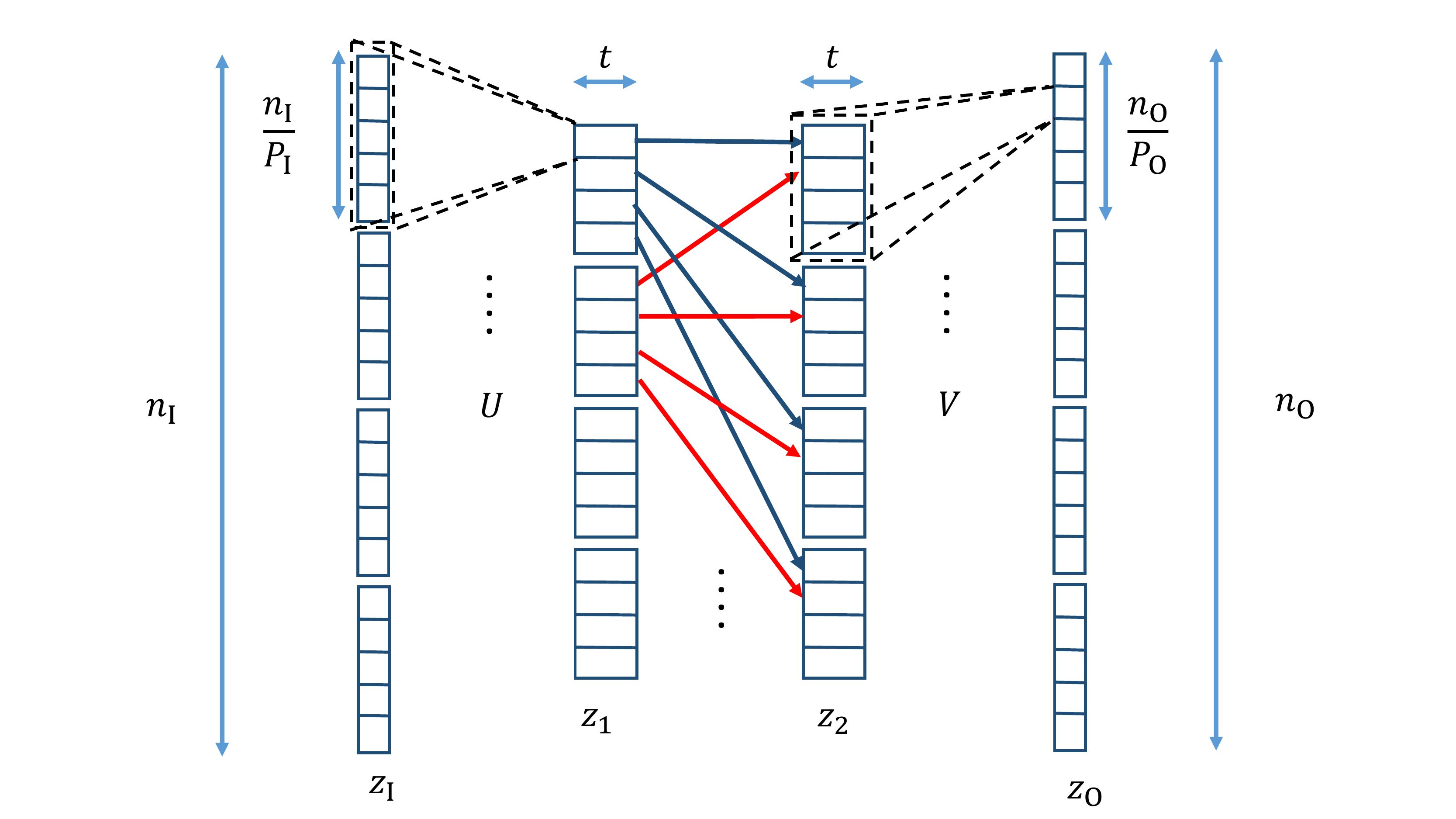}
  \caption{An illustration of the \texttt{Switch} layer where the detail description of it is
    provided in Section \ref{section: layer}. For the purpose of illustration we let
    $n_\inp=n_\out=20,P_\inp=P_\out=4$.}\label{figure: switch}
\end{figure}

{\bf Switch layer.} $z_\out = \texttt{Switch}[t,P_\inp,P_\out,n_\out](z_\inp)$ with input
$z_\inp\in\bbC^{n_\inp}$. It is assumed that $n_\inp$ and $n_\out$ are integer multiples of $P_\inp$
and $P_\out$, respectively. This layer consists the following steps.
\begin{itemize}
\item
  Apply $U^T$ to $z_\inp$:
  \begin{gather*}
	z_1 = U^T z_\inp \in \bbC^{P_\out P_\inp t},\\
	U^T = \begin{bmatrix} U_0^T & & \\ & \ddots &  \\ & & U_{P_\inp-1}^T \end{bmatrix},
    \quad U_0^T,\ldots,U_{P_\inp-1}^T \in \bbC^{ tP_\out\times \frac{n_\inp}{P_\inp}}
  \end{gather*}
\item 
  Reshape $z_1$ to be a $\bbC^{P_\out \times P_\inp \times t}$ tensor. Here we follow the Python
  convention of going through the dimensions from the last one to the first one.  Then a permutation
  is applied to swap the first two dimensions to obtain a tensor of size $\bbC^{P_\out \times
    P_\inp\times t}$.  Finally, the result is reshaped to a vector $z_2\in\bbC^{P_\inp
    P_\out t}$ again going through the dimensions from the last to the first.
\item
  Apply $V$ to $ z_2$:
  \begin{gather*}
	z_\out = V  z_2 \in \mathbb{R}^{n_\out},\cr
	V = \begin{bmatrix} {V_0} & & \\ & \ddots &  \\ & & {V_{P_\out-1}} \end{bmatrix},\quad
    V_0,\ldots,{V_{P_\out-1}} \in \bbC^{ \frac{n_\out}{P_\out} \times tP_\inp}.
  \end{gather*}
  Here the non-zero entries of $U, V$ are the trainable parameters. The Switch layer is illustrated
  in Figure \ref{figure: switch}.
\end{itemize}

{\bf Convolution layer.} $z_\out = \texttt{Conv}[w,c_\out](z_\inp)$ with input $z_\inp =\bbC^{n
  \times n\times c_\inp}$. Here $c_\inp, c_\out$ denote the input and output channel numbers and $w$
denotes the window size. In this paper we only use the convolution layer with stride 1 and with
zero-padding:
\begin{eqnarray}
  z_\out(k_1,k_2,k_3) &=&\text{ReLU}\Bigg(\sum_{l_1=\max(0,k_1-\frac{w-1}{2})}^{\min(n-1,k_1+\frac{w-1}{2})}\sum_{l_2=\max(0,k_2-\frac{w-1}{2})}^{\min(n-1,k_2+\frac{w-1}{2})} \sum_{l_3=0}^{c_\inp-1} \cr 
  &\  &\qquad   W\left(l_1-k_1+\frac{w-1}{2},l_2-k_2+\frac{w-1}{2},l_3,k_3\right) z_\inp(l_1,l_2,l_3)+b(k_3) \Bigg)
\end{eqnarray}
with $k_1,k_2 = 0,\ldots,n-1,\ k_3= 0,\ldots,c_\out-1$. Here $\text{ReLU}(x) = \max(0,x)$ and $w$ is
assumed to be odd in the presentation. Both $W\in \bbC^{w\times w \times c_\inp\times c_\out\ }$ and
$b\in\bbC^{c_\out}$ are the trainable parameters.

{\bf Pointwise multiplication layer.} $z_\out = \texttt{PM}(z_\inp)$ with input $z_\inp\in
\bbC^{n\times n\times c_\inp}$. It is defined as
\begin{equation} \label{item: lc}
  z_\out(k_1,k_2) = W(k_1,k_2) z_\inp(k_1,k_2) +b(k_1,k_2),  
\end{equation}
$k_1,k_2=0,\ldots,n-1$. Both $W\in \bbC^{n \times n}$ and $b\in\bbC^{n\times n}$ are trainable
parameters.

We remark that, among these layers, the \texttt{Switch} layer has the most parameters. If the input
and output to the \texttt{Switch} layer both have size $n\times n$, the number of parameter is
$2tPn^2$ where $P$ is the number of squares that partition the input field and $t$ is the rank of
the low-rank approximation.

\subsubsection{NN for the forward map $\eta \rightarrow d$.}\label{section: NNf1}


We move on to discuss the parameterization of the forward map $\eta \rightarrow d$. The proposal is
based on the simple observation that {\em the inverse of the inverse map is the forward map.}

More precisely, we simply reverse the architecture of the inverse map proposed in Algorithm
\ref{algorithm: inverse}. This results in an NN presented Algorithm \ref{algorithm: forward}. The
basic architecture of this NN involves applying a few layers of \texttt{Conv} first, then followed
by a \texttt{Switch} layer that mimics $A \approx U \Sigma V^*$.

\begin{algorithm}
  \caption{SwitchNet for the forward map $\eta \rightarrow d$ of far field pattern.}
  \label{algorithm: forward}
	\begin{algorithmic}[1]
		\Require $t,P_D, P_X, M, w, \alpha, L, \eta\in \mathbb{R}^{N\times N}$
		\Ensure $d\in \mathbb{R}^{M\times M\times 2}$
		\State $\eta_0 \leftarrow \eta$
		\For  {$\ell$ from $0$ to $L-1$}
		\State  $\eta_{\ell+1}\leftarrow \texttt{Conv}[w,\alpha](\eta_{\ell})$
		\EndFor
		\State $d_1\leftarrow \texttt{Conv}[w,1](\eta_{L})$
		\State $d_2\leftarrow \texttt{Vect}[P_X](d_1)$
		\State $d_3\leftarrow \texttt{Switch}[t,P_X, P_D, M^2](d_2)$
		\State $d \leftarrow \texttt{Square}[P_D](d_3)$
		\State \Return $d$
	\end{algorithmic}
\end{algorithm}

We would also like to mention yet another possibility to parameterize the forward map $\eta
\rightarrow d$, via a recurrent neural network \cite{mikolov2010recurrent}. Let
\begin{equation}
  E_\eff = E +  E G_0 E + E G_0 E G_0 E + \cdots =: E_1 + E_2 + E_3 + \cdots.
\end{equation}
One can leverage the following recursion 
\begin{equation}
  E_{k+1} = E G_0 E_{k} \quad k=1,2,\ldots K
\end{equation}
to approximate $E_\eff$ by treating each $E_k$ as an $N^2 \times N^2$ image and using a recurrent
neural network. At the $k$-th level of the recurrent neural network, it takes $E_k$ and $E$ as
inputs and outputs $E_{k+1}$. More specifically, in order to go from $E_k$ to $E_{k+1}$, one first
apply $G_0$ to each column of the image $E_k$, then each row of the image is reweighted by the
diagonal matrix $E$. Stopping at the $K$-th level for a sufficiently large $K$, $E_\eff$ can be
approximated by 
\begin{equation}
  E_\eff \approx \sum_{i=1}^{K+1} E_i.
\end{equation}
Once holding such an approximation to $E_\eff$, we plug it into \eqref{perturbation series}
\begin{equation}
  \label{double conv}
  d(r,s) = \sum_{x\in X} \sum_{y\in X} e^{i\omega r\cdot x} E_\eff(x,y) e^{-i\omega s\cdot y} =
  \sum_{x\in X} e^{i\omega r\cdot x} \left(\sum_{y\in X} E_\eff(x,y) e^{-i\omega s\cdot y} \right).
\end{equation}
This shows that the map from $E_\eff$ to $d$ can be realized by applying a matrix product to
$E_\eff$ first on the $y$-dimension, then on the $x$-dimension. If we view applying the Green's
function $G_0$ as applying a convolution layer in an NN, the above discussion shows that the forward
map can be obtained by first applying a recurrent NN followed by a convolutional NN. The main
drawback of this approach is the large memory requirement (i.e., $N^2 \times N^2$) to store each
individual $E_k$. In addition, the use of a recurrent NN may lead to difficulty in training
\cite{pascanu2013difficulty} due to the issue of exploding or vanishing gradient. Moreover, since
the weights for parameterizing $G_0$ are shared over multiple layers in the recurrent NN, one might
not be able to efficiently use back-propagation, which may lead to a longer training time.  These
are the main reasons why we decided to adopt the approach in Algorithm \ref{algorithm: forward}.

\subsection{Numerical results}\label{section: scat numeric}

In this section, we present numerical results of SwitchNet for far field pattern at a frequency
$\omega\approx 60$. The scatterer field $\eta(x)$ supported in $\Omega = [-0.5,0.5]^2$ is assumed to
be a mixture of Gaussians
\begin{equation}
  \label{gmm}
  \sum_{i=1}^{n_s}  \beta \exp\left(-\frac{\vert x-c_i \vert^2}{2\sigma^2} \right)
\end{equation}
where $\beta = 0.2$ and $\sigma = 0.015$. When preparing the training and testing examples, the
centers $\{c_i\}_{i=1}^{n_s}$ of the Gaussians are chosen to be uniformly distributed within
$\Omega$. The number $n_s$ of the Gaussians in the mixture is set to vary between $2$ and $4$.  In
the numerical experiments, the domain $\Omega = [-0.5,0.5]^2$ is discretized by an $80\times 80$
Cartesian grid $X$. To discretize the source and receiver directions, we set both $R$ and $S$ to be
a set of 80 equally spaced unit directions on $\mathbb{S}^1$. Therefore in this example, $N=M=80$.

In Algorithm \ref{algorithm: inverse}, the parameters are specified as $t=3$ (rank of the low-rank
approximation), $P_X = 8^2$, $P_D = 4^2$, $w=10$ (window size of the convolution layers), $\alpha =
18$ (channel number of the convolution layers), and $L=3$ (number of convolution layers), resulting
3100K number of parameters. The parameters for Algorithm \ref{algorithm: forward} are chosen to be
$t=4$, $P_X=8^2$, $P_D = 4^2$, $w=10$, $\alpha = 24$, and $L=3$, with a total of 4200K
parameters. Note that for both algorithms the number of parameters is significantly less than the
one of a fully connected NN, which has at least $80^4 = 40960$K parameters.

SwitchNet is trained with the ADAM optimizer \cite{kingma2014adam} in Keras \cite{chollet2017keras}
with a step size of 0.002 and a mini-batch size of size 200. The optimization is run for 2500
epochs. Both the training and testing data sets are obtained by numerically solving the forward
scattering problem with an accurate finite difference scheme with a perfectly matched layer. In the
experiment, 12.5K pairs of $(\eta,d)$ are used for training, and another 12.5K pairs are reserved
for testing. The errors are reported using the mean relative errors
\begin{equation}
  \frac{1}{N_\test}\sum_{i=1}^{N_\test} \frac{\| d^\NN_i - {d}_i \|_F}{\| {d}_i \|_F },\quad
  \frac{1}{N_\test}\sum_{i=1}^{N_\test} \frac{\| \eta^\NN_i - {\eta}_i \|_F}{\| {\eta}_i \|_F },
\end{equation}
where $d^\NN_i$ and $d_i$ denote the predicted and ground truth scattering patterns respectively for
the $i$-th testing sample, and $\eta^\NN_i$ and $\eta_i$ denote the predicted and ground truth
scatterer field respectively. Here $\|\cdot\|_F$ is the Frobenius norm.


Table \ref{table: scattering} summarizes the test errors for Gaussian mixtures with different
choices of $n_s$. For the purpose of illustration, we show the predicted $d$ and $\eta$ by SwitchNet
along with the ground truth in Figure \ref{figure: scat} for one typical test sample.

\begin{table}[!ht]
	\centering 
	\begin{tabular}{c c c } 
		\hline\hline 
		$n_s$& Forward map &  Inverse map\\ [0.5ex] %
		\hline\hline 
		2 & 9.4e-03 & 1.2e-02 \\
		3 & 4.0e-02 & 1.4e-02 \\
		4 & 4.8e-02 &  2.4e-02  \\
		\hline 
	\end{tabular}
	\caption{Prediction error of SwitchNet for the maps $\eta\rightarrow d$ and $d\rightarrow \eta$
      for far field pattern. }\label{table: scattering} 
\end{table}

\begin{figure}[!ht]
  \centering
  \subfloat[Ground truth scattering pattern $d$.]{\includegraphics[width=0.45\columnwidth]{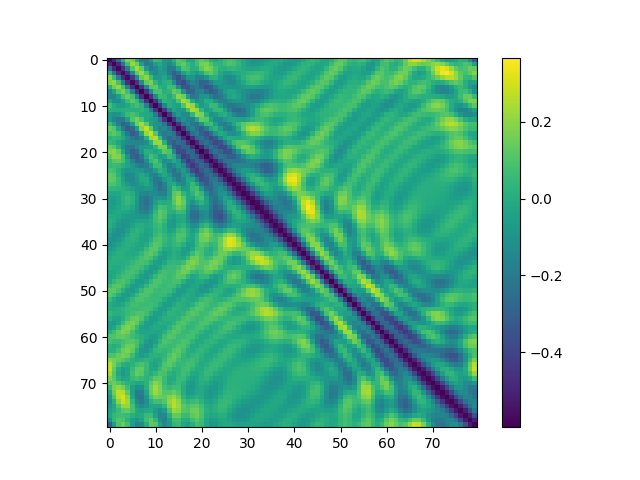}}
  \subfloat[Predicted scattering pattern $d^\NN$.]{\includegraphics[width=0.45\columnwidth]{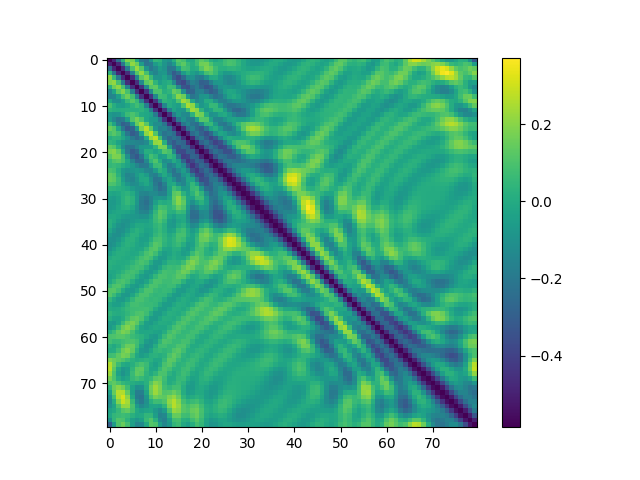}}
  \qquad
  \subfloat[Ground truth scatterers $\eta=\omega^2(1/c^2-1/c_0^2)$.]{\includegraphics[width=0.45\columnwidth]{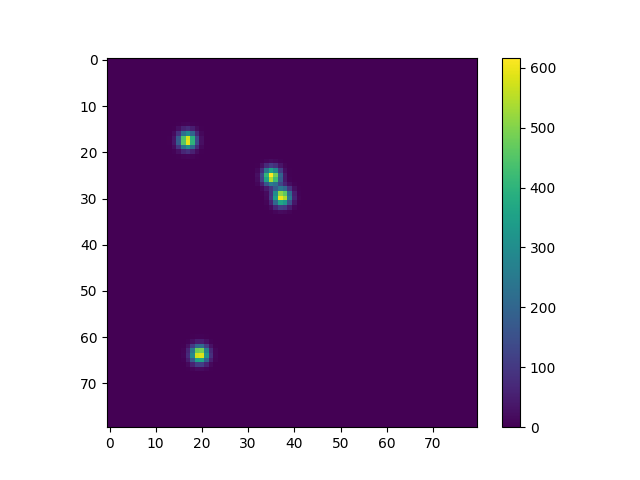}}
  \subfloat[Predicted scatterers $\eta^\NN$.]{\includegraphics[width=0.45\columnwidth]{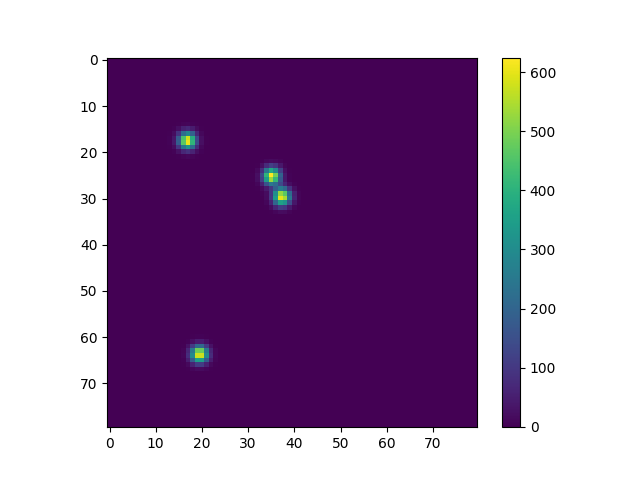}}
  \caption{Results for a typical instance of the far field pattern problem with $n_s=4$. (a) The
    ground truth scattering pattern. (b) The scattering pattern predicted by SwitchNet with a
    4.9e-02 relative error. (c) The ground truth scatterers. (d) The scatterers predicted by
    SwitchNet with a 2.4e-02 relative error.}\label{figure: scat}
\end{figure}

\section{SwitchNet for seismic imaging}\label{section: NN seis}

\subsection{Problem setup}

This section considers a two-dimensional model problem for seismic imaging. The scatterer $\eta(x)$
is again assumed to be supported in a domain $\Omega$ with an $O(1)$ diameter, after appropriate
rescaling. $\Omega$ is discretized with a Cartesian grid $X= \{x\}_{x\in X}$ at the rate of at least
a few point per wavelength. Compared to the source and receiver configurations in Section
\ref{section:FFsetup}, the experiment setup here is simpler.  One can regard both $S=\{s\}_{s\in S}$
and $R=\{r\}_{r\in R}$ to be equal to a set of uniformly sampled points along a horizontal line near
the top surface of the domain. The support of $\eta$ is at a certain distance below the top surface
so that it is well-separated from the sources and the receivers (see Figure \ref{fig:seismic} for an
illustration of this configuration).

\begin{figure}[!ht]
  \centering
  \includegraphics[height=0.25\textwidth]{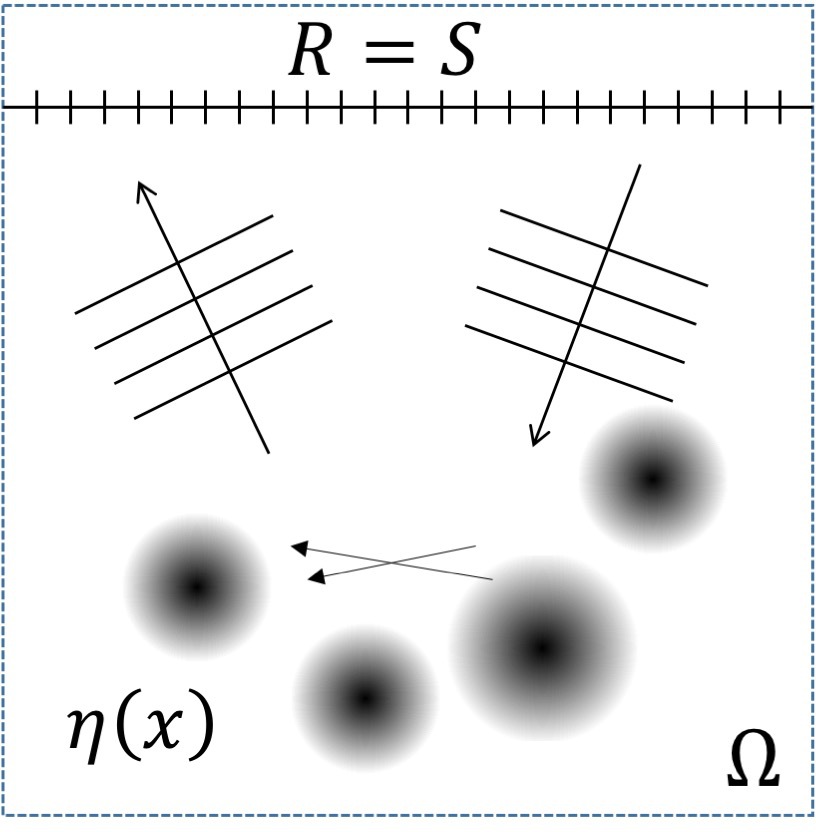}
  \caption{Illustration of a simple seismic imaging setting. The sources ($S$) and receivers ($R$)
    are located near the surface level (top) of the domain $\Omega$. The scatterer field $\eta(x)$ is
    assumed to be well-separated from the sources and the receivers.  }
  \label{fig:seismic}
\end{figure}



The source and receiver operators in \eqref{ddefined} take a particularly simple form. For the
sources, the operator $(G_0 \Pi_S)$ is simply given by sampling:
\[
(G_0 \Pi_S)(x,s) = G_0(x,s).
\]
Similarly for the receivers, the operator $(\Pi_R^T G_0)$ is given by
\[
(\Pi_R G_0)(r,y) = G_0(r,y).
\]
After plugging these two formulas back into \eqref{ddefined}, one arrives at the following
representation of the observation data $d(r,s)$ for $r\in R$ and $s\in S$
\[
d(r,s) = \sum_{x\in X} \sum_{y \in X} G_0(r,x) (E+EG_0E+\cdots)(x,y) G_0(y,s).
\]

\subsection{Low-rank property} 
\label{section: seismic low-rank}

Following the approach taken in Section \ref{section: FIO}, we start with the linear approximation
under the assumption that $\eta(x)$ is weak. Since $E=\text{diag}(\eta)$, the first order
approximation is
\begin{equation}
  d(r,s) \approx \sum_{x\in X} G_0(r,x)G_0(x,s) \eta(x),
\end{equation}
By regarding $\eta$ as a vector in $\mathbb{R}^{N^2}$ and $d$ as a vector $\in \mathbb{C}^{M^2}$,
one obtains the linear system
\begin{equation}
  \label{forward linear map}
  d \approx A\eta,\quad A\in\mathbb{C}^{M^2 \times N^2},
\end{equation}
where the element $A$ at $(r,s)\in R\times S$ and $x\in X$ is given by 
\begin{equation}
  A(rs,x) = G_0(r,x)G_0(x,s) = G_0(r,x)G_0(s,x) .
\end{equation}

Under the assumptions that the sources $S$ and receivers $R$ are well-separated from the support of
$\eta(x)$ and that $c_0(x)$ varies smoothly, the matrix $A$ satisfies a low-rank property similar to
Theorem \ref{theorem: low-rank}. To see this, we again partition $X$ into Cartesian squares
$X_0,\ldots, X_{P_X-1}$ of side-length equal to $1/\sqrt{\omega}$. Since $R=S$ is now the
restriction of $X$ on the surface level, this partition also induces a partitioning for $R\times
S$. When $c_0(x)$ varies smoothly, it is shown (see for example \cite{Engquist2018}) that the
restriction of the matrix $[G_0(r,x)]_{r\in R, x\in X}$ (or $[G_0(s,x)]_{s\in S, x\in X}$) to each
piece of the partitioning is numerically low-rank. Since the matrix $A$ is obtained by taking the
Khatri-Rao product \cite{liu2008hadamard} of $[G_0(r,x)]_{r\in R, x\in X}$, $[G_0(s,x)]_{s\in S,
  x\in X}$, the low-rank property is preserved with the guarantee that the rank at most squares in
the worst case.

By following the same argument in Section \ref{section: fac}, one can show that the matrix $A$ has a
low-complexity matrix factorization $ A \approx U \Sigma V^* $ of exactly the same structure as
\eqref{butterfly mid}. The corresponding factorization for $A^*$ is $A^* \approx V \Sigma^* U^*$.

\subsection{Neural networks}

	
	

Based on the low-rank property in Section \ref{section: seismic low-rank}, we propose here SwitchNet
for seismic imaging. 

\subsubsection{NN for the inverse map $d \rightarrow \eta$}

When the linear approximation is valid (i.e., \eqref{forward linear map} holds) $\eta$ can be
obtained from $d$ via a filtered projection approach (or called migration in the seismic community)
\begin{equation}
  \eta \approx (A^* A + \eps I)^{-1} A^* d,
\end{equation}
where $\eps I$ is a regularizing term. Since $A^*$ has a low-complexity factorization $A^*\approx
V\Sigma^* U^*$, the application $A^*$ to a vector can be represented by a $\texttt{Switch}$ layer.

Concerning the $(A^* A + \eps I)^{-1}$ term, note that
\begin{equation}
  \label{noninvar kernel}
  (A^*A)(x,y) \approx \sum_{rs\in R\times S} \overline{G_0(x,r)G_0(x,s)} G_0(r,y)G_0(s,y),
\end{equation}
which, unlike \eqref{full kernel}, is no longer a translation-invariant kernel as the data gathering
setup is not so. For example, even when the background velocity $c_0(x)=1$, the different terms of
the Green's function $G_0(\cdot)$ in \eqref{noninvar kernel} scale like
\[
\frac{1}{\sqrt{\vert x - r\vert}}, \quad
\frac{1}{\sqrt{\vert x - s\vert}}, \quad
\frac{1}{\sqrt{\vert y - r\vert}}, \quad
\frac{1}{\sqrt{\vert x - s\vert}},
\]
which fail to give a translation-invariant kernel of form $K(x-y)$. As a direct consequence, the
operator $(A^* A + \eps I )^{-1}$ is not translation-invariant either.

In order to capture the loss of translation-invariance, we include an extra pointwise multiplication
layer $\texttt{PM}$ (defined in Section \ref{section: layer}) when dealing with the inverse map. The
pseudo-code of the NN for the inverse map is given in Algorithm \ref{algorithm: inverse2}.



\begin{algorithm}
  \caption{SwitchNet for the inverse map $d\rightarrow \eta$ of seismic imaging.}\label{algorithm:
    inverse2}
  \begin{algorithmic}[1]
	\Require $t,P_D,P_X,N,w,\alpha, L, d\in \bbC^{M\times M}$
	\Ensure $\eta\in \mathbb{R}^{N\times N}$
	\State $d_1\leftarrow  \texttt{Vect}[P_D](d)$
	\State $d_2\leftarrow  \texttt{Switch}[t,P_D,P_X,N^2](d_1)$
	\State $e_0\leftarrow \texttt{Square}[P_X](d_2)$
	\For  {$\ell$ from $0$ to $L-1$}
	\State  $e_{\ell+1}\leftarrow \texttt{Conv}[w,\alpha](e_{\ell})$
	\EndFor
	\State $\eta \leftarrow \texttt{Conv}[w,1](e_{L})$
	\State $\eta \leftarrow \texttt{PM}(\eta)$
	\State \Return $\eta$
  \end{algorithmic}
\end{algorithm}

\subsubsection{NN for the forward map $\eta \rightarrow d$}

As in Section \ref{section: NNf1}, for the forward map from $\eta \rightarrow d$, we simply reverse
the architecture of the NN for the inverse map in Algorithm \ref{algorithm: inverse2}. For
completeness we detail its structure in Algorithm \ref{algorithm: forward2}. The main difference
between Algorithm \ref{algorithm: forward} and Algorithm \ref{algorithm: forward2} is again the
inclusion of an extra pointwise multiplication layer.
\begin{algorithm}
  \caption{SwitchNet for the forward map $\eta\rightarrow d$ of seismic imaging.}\label{algorithm:
    forward2}
	\begin{algorithmic}[1]
		\Require $t,P_D, P_X, M, w,\alpha,L,\eta\in \bbC^{N\times N}$
		\Ensure $d\in \bbC^{M\times M}$
		\State $\eta_0 \leftarrow \texttt{PM}(\eta)$
		\For  {$\ell$ from $0$ to $L-1$}
		\State  $\eta_{\ell+1}\leftarrow \texttt{Conv}[w,\alpha](\eta_{\ell})$
		\EndFor
		\State $d_1 \leftarrow \texttt{Conv}[w,1](\eta_{L})$
		\State $d_2 \leftarrow \texttt{Vect}[P_X](d_1)$
		\State $d_3 \leftarrow \texttt{Switch}[t,P_X, P_D, M^2](d_2)$.
		\State $d \leftarrow \texttt{Square}[P_D](d_3)$
		\State \Return $d$
	\end{algorithmic}
\end{algorithm}

\subsection{Numerical results}

In the numerical experiments, we set $\Omega = [-0.5,0.5]^2$ and discretize it by a $64\times 64$
Cartesian grid. As mentioned before, the sources $S$ and the receivers $R$ are located on a line
near the top surface of $\Omega$, similar to the setting in Fig. \ref{fig:seismic}. This line is
discretized uniformly with $M=80$ points. Therefore, the size of $\eta$ and $d$ are $64\times 64$
and $80\times 80$, respectively. We assume a Gaussian mixture model for $\eta$ as in \eqref{gmm},
where $\beta=0.2,\sigma=0.015$. Unlike before, the centers $\{c_i\}_{i=1}^{n_s}$ are kept away from
the top surface of $\Omega$ in order to ensure that they are well-separated from the sources and
receivers.

In Algorithm \ref{algorithm: inverse2} and Algorithm \ref{algorithm: forward2}, the parameters are
set to be $t=3$, $P_X=8^2$, $P_D = 4^2$, $N=64$, $M=80$, $w=8$, $\alpha=18$, and $L=3$, resulting
NNs with 2900K parameters. The procedure of training the NNs is the same as the one used in Section
\ref{section: scat numeric}. Table \ref{table: seismic} presents the test errors for this model
problem. The predicted and the ground truth $d,\eta$ are visually compared in Figure \ref{figure:
  seismic} for one typical test sample.

\begin{table}[ht]
  \centering 
  \begin{tabular}{c c c } 
	\hline\hline 
	$n_s$& Forward map &  Inverse map\\ [0.5ex] %
	\hline\hline 
	2 & 5.6e-02 & 2.1e-02 \\
	3 & 7.7e-02 & 2.2e-02 \\
		4 & 8.0e-02 &  5.1e-02  \\
		\hline 
  \end{tabular}
  \caption{Prediction error of SwitchNet for the maps $\eta\rightarrow d$ and $d\rightarrow \eta$
    for seismic imaging.  }
  \label{table: seismic}
\end{table}

\begin{figure}[!ht]
  \centering
  \subfloat[Ground truth scattering pattern $d$.]{\includegraphics[width=0.45\columnwidth]{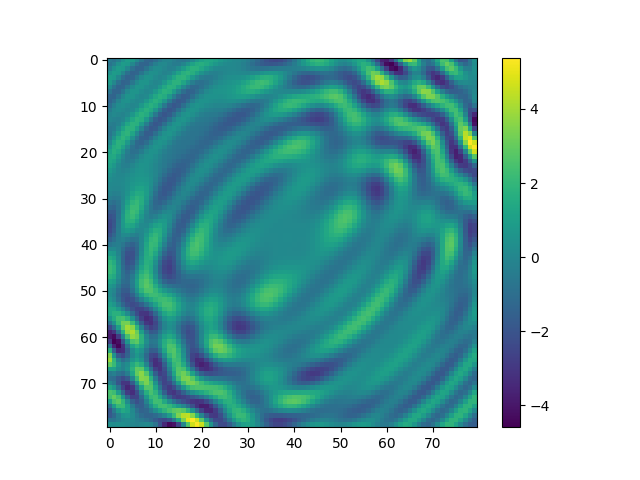}}
  \subfloat[Predicted scattering pattern $d^\NN$.]{\includegraphics[width=0.45\columnwidth]{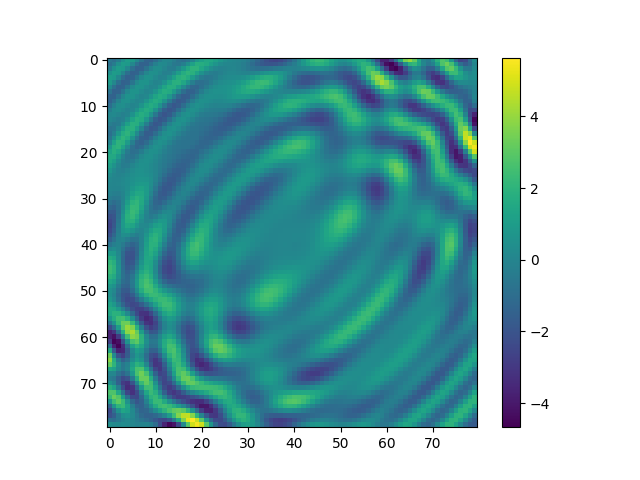}}
  \qquad
  \subfloat[Ground truth scatterers $\eta=\omega^2(1/c^2-1/c_0^2)$.]{\includegraphics[width=0.45\columnwidth]{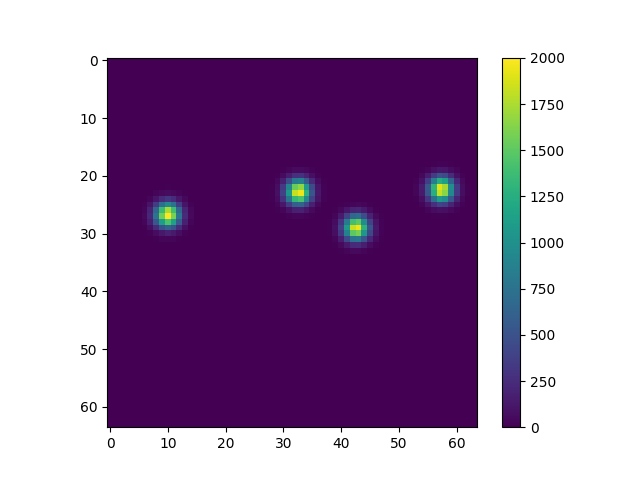}}
  \subfloat[Predicted scatterers $\eta^\NN$.]{\includegraphics[width=0.45\columnwidth]{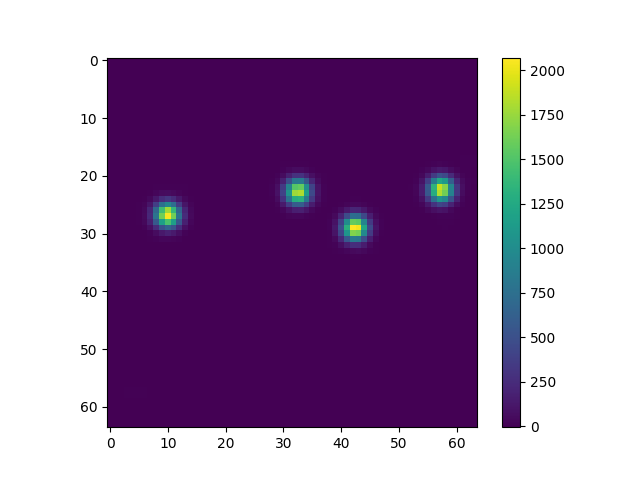}}
  \caption{Results for a typical instance of the seismic imaging setting with $n_s=4$. (a) The
    ground truth scattering pattern. (b) The scattering pattern predicted by SwitchNet with a
    7.7e-02 relative error. (c) The ground truth scatterers. (d) The scatterers predicted by
    SwitchNet with a 6.9e-02 relative error.}\label{figure: seismic}
\end{figure}

\section{Discussion}
In this paper, we introduce a neural network, SwitchNet, for approximating forward and inverse maps
arising from the time-harmonic wave equation. For these maps, local information at the input has a
global impact at the output, therefore they generally require the use of a fully connected NN in
order to parameterize them. Based on certain low-rank property that arises in the linearized
operators, we are able to replace a fully connected NN with the sparse SwitchNet, thus reducing
complexity dramatically. Furthermore, unlike convolutional NNs with local filters, the proposed
SwitchNet connects the input layer with the output layer globally. This enables us to represent
highly oscillatory wave field resulted from scattering problems, and to solve for the associated
inverse problems.


\section*{Acknowledgments}

The work of Y.K. and L.Y. is partially supported by the U.S. Department of Energy, Office of
Science, Office of Advanced Scientific Computing Research, Scientific Discovery through Advanced
Computing (SciDAC) program and the National Science Foundation under award DMS-1818449.  Y.K. thanks
Prof. Emmanuel Cand\`es for the partial support from a Math+X postdoctoral fellowship. This work is
also supported by the GCP Research Credits Program from Google.

\bibliographystyle{abbrv}
\bibliography{bibref}

\begin{thebibliography}{10}

\bibitem{berenger1994perfectly}
J.-P. Berenger.
\newblock A perfectly matched layer for the absorption of electromagnetic
  waves.
\newblock {\em Journal of computational physics}, 114(2):185--200, 1994.

\bibitem{candes2009fast}
E.~Cand{\`e}s, L.~Demanet, and L.~Ying.
\newblock A fast butterfly algorithm for the computation of {F}ourier integral
  operators.
\newblock {\em Multiscale Modeling \& Simulation}, 7(4):1727--1750, 2009.

\bibitem{carleo2017solving}
G.~Carleo and M.~Troyer.
\newblock Solving the quantum many-body problem with artificial neural
  networks.
\newblock {\em Science}, 355(6325):602--606, 2017.

\bibitem{chollet2017keras}
F.~Chollet.
\newblock Keras (2015).
\newblock {\em URL http://keras. io}, 2017.

\bibitem{Colton2013}
D.~Colton and R.~Kress.
\newblock {\em Inverse acoustic and electromagnetic scattering theory},
  volume~93 of {\em Applied Mathematical Sciences}.
\newblock Springer, New York, third edition, 2013.

\bibitem{weinan2018deep}
W.~E and B.~Yu.
\newblock The deep {R}itz method: A deep learning-based numerical algorithm for
  solving variational problems.
\newblock {\em Communications in Mathematics and Statistics}, 6(1):1--12, 2018.

\bibitem{Engquist2018}
B.~Engquist and H.~Zhao.
\newblock Approximate separability of the green's function of the {H}elmholtz
  equation in the high frequency limit.
\newblock {\em Communications on Pure and Applied Mathematics},
  71(11):2220--2274, 2018.

\bibitem{fan2018multiscale}
Y.~Fan, L.~Lin, L.~Ying, and L.~Zepeda-N{\'u}nez.
\newblock A multiscale neural network based on hierarchical matrices.
\newblock {\em Arxiv preprint arXiv:1807.01883}, 2018.

\bibitem{goodfellow2016deep}
I.~Goodfellow, Y.~Bengio, and A.~Courville.
\newblock {\em Deep learning}.
\newblock MIT press, 2016.

\bibitem{han2018solving}
J.~Han, A.~Jentzen, and E.~Weinan.
\newblock Solving high-dimensional partial differential equations using deep
  learning.
\newblock {\em Proceedings of the National Academy of Sciences}, page
  201718942, 2018.

\bibitem{han2017deep}
J.~Han, L.~Zhang, R.~Car, et~al.
\newblock Deep potential: A general representation of a many-body potential
  energy surface.
\newblock {\em Arxiv preprint arXiv:1707.01478}, 2017.

\bibitem{hinton2006reducing}
G.~E. Hinton and R.~R. Salakhutdinov.
\newblock Reducing the dimensionality of data with neural networks.
\newblock {\em Science}, 313(5786):504--507, 2006.

\bibitem{khoo2017solving}
Y.~Khoo, J.~Lu, and L.~Ying.
\newblock Solving parametric pde problems with artificial neural networks.
\newblock {\em Arxiv preprint arXiv:1707.03351}, 2017.

\bibitem{Khoo2018}
Y.~Khoo, J.~Lu, and L.~Ying.
\newblock Solving for high dimensional committor functions using artificial
  neural networks.
\newblock {\em Arxiv preprint arXiv:1802.10275}, 2018.

\bibitem{kingma2014adam}
D.~Kingma and J.~Ba.
\newblock Adam: A method for stochastic optimization.
\newblock {\em Arxiv preprint arXiv:1412.6980}, 2014.

\bibitem{lagaris1998artificial}
I.~E. Lagaris, A.~Likas, and D.~I. Fotiadis.
\newblock Artificial neural networks for solving ordinary and partial
  differential equations.
\newblock {\em IEEE Transactions on Neural Networks}, 9(5):987--1000, 1998.

\bibitem{Larsson2009}
S.~Larsson and V.~Thom\'{e}e.
\newblock {\em Partial differential equations with numerical methods},
  volume~45 of {\em Texts in Applied Mathematics}.
\newblock Springer-Verlag, Berlin, 2009.
\newblock Paperback reprint of the 2003 edition.

\bibitem{lecun2015deep}
Y.~LeCun, Y.~Bengio, and G.~Hinton.
\newblock Deep learning.
\newblock {\em Nature}, 521(7553):436--444, 2015.

\bibitem{li2018butterfly}
Y.~Li, X.~Cheng, and J.~Lu.
\newblock Butterfly-net: Optimal function representation based on convolutional
  neural networks.
\newblock {\em Arxiv preprint arXiv:1805.07451}, 2018.

\bibitem{li2017interpolative}
Y.~Li and H.~Yang.
\newblock Interpolative butterfly factorization.
\newblock {\em SIAM Journal on Scientific Computing}, 39(2):A503--A531, 2017.

\bibitem{li2015butterfly}
Y.~Li, H.~Yang, E.~R. Martin, K.~L. Ho, and L.~Ying.
\newblock Butterfly factorization.
\newblock {\em Multiscale Modeling \& Simulation}, 13(2):714--732, 2015.

\bibitem{liu2008hadamard}
S.~Liu and G.~Trenkler.
\newblock {H}adamard, {K}hatri-{R}ao, {K}ronecker and other matrix products.
\newblock {\em Int. J. Inf. Syst. Sci}, 4(1):160--177, 2008.

\bibitem{long2017pde}
Z.~Long, Y.~Lu, X.~Ma, and B.~Dong.
\newblock {PDE}-net: Learning {PDE}s from data.
\newblock {\em Arxiv preprint arXiv:1710.09668}, 2017.

\bibitem{mikolov2010recurrent}
T.~Mikolov, M.~Karafi{\'a}t, L.~Burget, J.~{\v{C}}ernock{\`y}, and
  S.~Khudanpur.
\newblock Recurrent neural network based language model.
\newblock In {\em Eleventh Annual Conference of the International Speech
  Communication Association}, 2010.

\bibitem{Natterer2001}
F.~Natterer.
\newblock {\em The mathematics of computerized tomography}, volume~32 of {\em
  Classics in Applied Mathematics}.
\newblock Society for Industrial and Applied Mathematics (SIAM), Philadelphia,
  PA, 2001.
\newblock Reprint of the 1986 original.

\bibitem{pascanu2013difficulty}
R.~Pascanu, T.~Mikolov, and Y.~Bengio.
\newblock On the difficulty of training recurrent neural networks.
\newblock In {\em International Conference on Machine Learning}, pages
  1310--1318, 2013.

\bibitem{rudd2015constrained}
K.~Rudd and S.~Ferrari.
\newblock A constrained integration ({CINT}) approach to solving partial
  differential equations using artificial neural networks.
\newblock {\em Neurocomputing}, 155:277--285, 2015.

\bibitem{schmidhuber2015deep}
J.~Schmidhuber.
\newblock Deep learning in neural networks: An overview.
\newblock {\em Neural networks}, 61:85--117, 2015.

\bibitem{ying2009sparse}
L.~Ying.
\newblock Sparse fourier transform via butterfly algorithm.
\newblock {\em SIAM Journal on Scientific Computing}, 31(3):1678--1694, 2009.

\end{thebibliography}

\end{document}